\documentclass[10pt,centertags,reqno]{amsart}

\usepackage{amsmath}
\usepackage{amsthm}
\usepackage{amssymb}
\usepackage{mathrsfs}
\usepackage{exscale}
\usepackage[arrow, matrix, curve]{xy}

\numberwithin{equation}{section}

\allowdisplaybreaks

\theoremstyle{definition}
	\newtheorem{definition}{Definition}
	\newtheorem*{definition*}{Definition}
	
	\numberwithin{definition}{section}
\theoremstyle{plain}
	\newtheorem{lemma}[definition]{Lemma}
	\newtheorem{proposition}[definition]{Proposition}
	\newtheorem{theorem}{Theorem}
	\newtheorem*{theorem*}{Theorem}
	\newtheorem{auxthm}[definition]{Theorem}
	\newtheorem{corollary}[definition]{Corollary}

	\newtheorem*{claim*}{Claim}
\theoremstyle{remark}

	\newtheorem*{conjecture}{Conjecture}

\DeclareMathOperator{\dop}{d}

\DeclareMathOperator{\xop}{\times}

\renewcommand{\bf}{\textbf}

\newcommand{\al}{\alpha}

\newcommand{\Ga}{\Gamma}

\newcommand{\io}{\iota}

\newcommand{\om}{\omega}
\newcommand{\Om}{\Omega}

\newcommand{\sig}{\sigma}

\renewcommand{\th}{\theta}

\newcommand{\vphi}{\varphi}

\newcommand{\C}{\mathbb{C}}

\newcommand{\R}{\mathbb{R}}

\newcommand{\Xca}{\mathcal{X}}

\newcommand{\gfr}{\mathfrak{g}}
\newcommand{\hfr}{\mathfrak{h}}
\newcommand{\kfr}{\mathfrak{k}}
\newcommand{\pfr}{\mathfrak{p}}

\newcommand{\map}[3]{#1\colon#2\rightarrow#3}
\newcommand{\mapin}[3]{#1\colon#2\hookrightarrow#3}
\newcommand{\lmap}[3]{#1\colon#2\lto#3}
\newcommand{\xto}{\xrightarrow}
\newcommand{\lto}{\longrightarrow}

\newcommand{\deq}{\mathrel{\mathop:}=}

\newcommand{\jump}{\hspace{5mm}}

\newcommand{\lmapcong}[3]{#1\colon#2\xrightarrow{\quad\cong\quad}#3}

\newcommand{\X}{\Xca}
\renewcommand{\l}{\langle}
\renewcommand{\r}{\rangle}

\renewcommand{\L}{\mathfrak}

\renewcommand{\L}{\mathfrak}

\begin{document}

\title[Surjectivity of comparison map]{Surjectivity of the comparison map in bounded cohomology for Hermitian Lie groups}

\author[T. Hartnick]{Tobias Hartnick}
\address{Departement Mathematik, ETH Z\"urich, R\"amistrasse 101, 8092 Z\"urich, Switzerland}
\email{hartnick@math.ethz.ch}

\author[A. Ott]{Andreas Ott}
\address{Departement Mathematik, ETH Z\"urich, R\"amistrasse 101, 8092 Z\"urich, Switzerland}
\email{andreas@math.ethz.ch}

\thanks{The first author was partially supported by SNF grant PP002-102765. The second author was supported by ETH Research Grant TH-01 06-1.}


\begin{abstract}
	We prove surjectivity of the comparison map from continuous bounded cohomology to continuous cohomology for Hermitian Lie groups with finite center. For general semisimple Lie groups with finite center, the same argument shows that the image of the comparison map contains all the even generators. Our proof uses a Hirzebruch type proportionality principle in combination with Gromov's results on boundedness of primary characteristic classes and classical results of Cartan and Borel on the cohomology of compact homogeneous spaces.
\end{abstract}

\maketitle

\section{Introduction}

Let $G$ be a connected semisimple Lie group without compact factors. $G$ is called \emph{Hermitian} if the symmetric space associated with $G$ admits a $G$-invariant complex structure. We denote by $H^\bullet_{cb}(G; \R)$ the continuous bounded cohomology ring of $G$ and by $H^\bullet_{c}(G; \R)$ the continuous cohomology ring of $G$, both with trivial real coefficients (see \cite{BC,BoWa} or Sections~\ref{SubsecExistenceUniversal}~and~\ref{SubsecGromov} for the definitions). These two rings are related by a natural \emph{comparison map} $H^\bullet_{cb}(G; \R) \to H^\bullet_{c}(G; \R)$ \cite[Def.\,9.2.1]{BC}. The purpose of this article is to prove the following result.

\begin{theorem} \label{MainThm}
	Let $G$ be a connected semisimple Lie group without compact factors and with finite center. If $G$ is Hermitian, then the comparison map
$H^\bullet_{cb}(G;\R) \to H^\bullet_{c}(G;\R)$ is surjective. 
\end{theorem}

Our methods still apply if the condition that $G$ be Hermitian is dropped, albeit with a weaker conclusion. In this case one has to distinguish between even and odd generators of $H^\bullet_{c}(G; \R)$ (see Section~\ref{SecOverview} for the definitions).

\begin{theorem} \label{MainThm2}
	Let $G$ be a semisimple Lie group with finite center and without compact factors. Then the subring of $H^\bullet_c(G; \R)$ generated by the even generators consists of bounded classes.
\end{theorem}

We will see in Section~\ref{SecOverview} that Theorem~\ref{MainThm2} implies Theorem~\ref{MainThm}.

Before we turn to the proofs of the theorems, let us explain their context. Continuous bounded cohomology of locally compact groups was introduced by Burger and Monod \cite{BuMo}  as a tool to compute the bounded cohomology groups in the sense of Gromov \cite{Gromov} of compact locally symmetric spaces of the non-compact type, and has found applications beyond that purpose in recent years \cite{MonodSurvey,Surface}. Burger and Monod obtained a complete understanding of continuous bounded cohomology of connected Lie groups in degree $2$. More precisely, they showed the following.
\begin{itemize} \itemsep 0.5ex
	\item[(i)] If $G$ is a connected Lie group with radical $R(G)$, then \[H^2_{cb}(G; \R)
\cong H^2_{cb}(G/R(G); \R).\] This reduces the computation of $H^2_{cb}(G; \R)$ to the case of semisimple Lie groups with finite center. Similarly, one can eliminate compact factors.
	\item[(ii)] If $G$ is a connected semisimple Lie group without compact factors and with finite center, then the degree $2$ comparison map $H^2_{cb}(G; \R) \to H^2_{c}(G; \R)$ is an isomorphism. Since $H^2_{c}(G; \R)$ is well-known, this allows one to compute $H^2_{cb}(G; \R)$ for arbitrary connected Lie groups.
\end{itemize}
One would like to prove a similar result for higher degree bounded cohomology groups. While the reduction step (i) is based on amenability methods and hence works in arbitrary degree, the key ingredient in the proof of step (ii) is double ergodicity, which does not have any analog in higher degrees. A higher degree generalization of step (ii) would therefore require methods different from those used by Burger and Monod. Nevertheless, it is commonly believed that such a higher degree generalization exists \cite[Conjecture 16.1]{Guido}, \cite{MonodSurvey}.

\begin{conjecture} \label{MainConj}
	If $G$ is a connected semisimple Lie group without compact factors and with finite center, then the comparison map $H^\bullet_{cb}(G; \R) \to H^\bullet_{c}(G; \R)$ is an isomorphism.
\end{conjecture}

While the injectivity part of this conjecture is rather new (apparently first suggested in \cite{MonodSurvey}), the surjectivity part goes back to a $30$ years old question of Dupont \cite[Remark\,3]{Dupont}. In fact, Dupont even suggests an explicit candidate for a bounded cocycle in a given cohomology class, namely the unique cocycle in the image of the van Est map. This stronger form of the conjecture was verified by Dupont himself in degree $2$ \cite{Dupont}, but is still open in higher degree. Theorem~\ref{MainThm} may be regarded as a positive answer to the cohomological version of Dupont's conjecture for Hermitian Lie groups.

\medskip

In Section~\ref{SecOverview} we give an outline of the proofs of Theorem~\ref{MainThm} and Theorem~\ref{MainThm2}. The organization of this article will be described at the end of that section.

\bigskip

\bf{Acknowledgments.} We are grateful to Marc Burger for pointing out reference \cite{Borel} and, in particular, Proposition~\ref{CorBorel} to us. While working on the proof of Theorem~\ref{MainThm} we learned from Michelle Bucher-Karlsson that a similar theorem should be true for certain even-degree classes in arbitrary semisimple Lie groups. We are indepted to her for this suggestion, which led us to discover Theorem~\ref{MainThm2}. The second author would like to thank the Department of Mathematics at Rutgers University for their hospitality and excellent working conditions.

\section{Outline of the proof} \label{SecOverview}

In this section we outline the proofs of Theorem~\ref{MainThm} and Theorem~\ref{MainThm2}.

Let $G$ be an arbitrary semisimple Lie group with finite center and without compact factors. The main technical result of this article is Proposition~\ref{MainConvenient} below, which exhibits an identification of two a priorily different maps between the singular cohomology $H^\bullet(BG;\R)$ of the classifying space $BG$ of $G$ and the continuous group cohomology $H^\bullet_c(G;\R)$ of $G$ (see Section \ref{SubsecExistenceUniversal} for the definition of continuous group cohomology).

The first of these maps is characterized by a universal property and will therefore be referred to as the \emph{universal map}
\[
\lmap{\sigma_G}{H^\bullet(BG;\R)}{H^\bullet_c(G;\R)}.
\]
In order to define $\sigma_G$, we first recall that for any discrete group $\Gamma$ there is a natural isomorphism $\map{\sigma_\Gamma}{H^\bullet(B\Gamma;\R)}{H^\bullet(\Gamma;\R)}$ (see Section~\ref{SubsecExistenceUniversal}). Then $\sigma_G$ is defined to be the unique extension of the family $\{\sigma_\Gamma\}$ to a natural transformation, that is, the unique map such that for every discrete group $\Gamma$ and every representation $\map{\rho}{\Gamma}{G}$, the diagram
\[
\begin{xy}\xymatrix{
	H^\bullet(BG;\R) \ar[d]_{(B\rho)^*} \ar[r]^{\sigma_G} & H^\bullet_c(G;\R) \ar[d]^{\rho^*} \\
	H^\bullet(B\Gamma;\R) \ar[r]^{\sigma_\Gamma} & H^\bullet(\Gamma;\R) \\
}\end{xy}
\]
commutes. We will prove in Section~\ref{SubsecExistenceUniversal} that the universal map $\sig_{G}$ actually exists. It is related to the question of boundedness of classes in the continuous cohomology of $G$ by the next proposition, which is an immediate consequence of Gromov's result on boundedness of primary characteristic classes \cite{Gromov} and will be proved in Section~\ref{SubsecGromov}.

\begin{proposition} \label{GromovConvenient}
	The image of the universal map $\sigma_G: H^\bullet(BG; \R) \to H^\bullet_c(G; \R)$ is contained in the image of the comparison map $H^\bullet_{cb}(G;\R) \to H^\bullet_c(G;\R)$.
\end{proposition}

We shall compare the universal map $\sig_{G}$ to the \emph{geometric map}
\[
\lmap{T_G}{H^\bullet(BG;\R)}{H^\bullet_c(G;\R)}
\]
that is defined in terms of explicit geometric data in the following way. Recall that $H^\bullet_c(G; \R)$ gets identified via the van Est isomorphism with the singular cohomology of the compact dual symmetric space $\X_u$ of $G$ (see Section \ref{SubsecVanEst}). Here $\X_u = G_u/K$, where $K$ is a maximal compact subgroup of $G$, and $G_u$ is the compact dual group of $G$ (see Section \ref{SubsecAuxiliary}). Then the geometric map is defined by
\begin{eqnarray} \label{T_Gexplicit}
	T_G\!: H^\bullet(BG;\R) \cong  H^\bullet(BK;\R)  \xto{f_{G_u}^*} H^\bullet(\X_u; \R) \cong H^\bullet_c(G;\R)
\end{eqnarray}
in terms of the classifying map $\map{f_{G_u}}{\X_u}{BK}$ of the canonical $K$-bundle $G_{u} \to \X_{u}$. The next proposition provides the desired identification between the universal and the geometric map.

\begin{proposition} \label{MainConvenient}
	Let $G$ be an arbitrary semisimple Lie group with finite center and without compact factors. Then the universal map $\sigma_G: H^\bullet(BG; \R) \to H^\bullet_c(G; \R)$ and the geometric map $T_G: H^\bullet(BG; \R) \to H^\bullet_c(G; \R)$ agree up to sign.
\end{proposition}

We will prove Proposition \ref{MainConvenient} in Section \ref{SectionExplicitDescription}.

We now come to the proof of Theorem~\ref{MainThm2}. Let us denote by $\map{f_{G_u}}{\X_u}{BK}$ the classifying map of the canonical $K$-bundle $\map{p_{G_u}}{G_u}{\X_u}$ over the compact dual symmetric space of $G$. By \cite[Sec.\,10]{Cartan2}, there is an isomorphism of algebras
\[
H^\bullet (\X_u; \R)\cong f_{G_u}^* H^\bullet(BK;\R) \otimes  p_{G_u}^* H^\bullet (\X_u; \R)
\]
which intertwines $p_{G_u}^*$ with the projection onto the second factor. We obtain from this identification generators for the cohomology ring $H^\bullet_c(G; \R) \cong H^\bullet(\X_u; \R)$ in the following way. The singular cohomology ring $H^\bullet_{sing}(K;\R)$ is a Hopf algebra and is therefore generated by odd degree primitive elements \cite{Hopf41}. These primitive elements give rise to even degree generators of $H^\bullet(BK;\R)$ via trangression in the universal $K$-bundle. Taking only those generators of $H^\bullet(BK;\R)$ that are mapped non-trivially under $f_{G_u}^*$ we then obtain generators of the first factor of $H^\bullet (\X_u; \R)$. These generators are unique up to real multiples and we shall refer to them as \emph{even generators} of $H^\bullet (\X_u;\R)$. The second factor $p_{G_u}^* H^\bullet (\X_u;\R) \subset H^\bullet_{sing}(G_u;\R)$ is generated by certain primitive elements lying in the image of the trangression map of the universal $G_u$-bundle. We will call these primitive elements the \emph{odd generators} of $H^\bullet (\X_u;\R)$. (We refer the reader to \cite{Cartan2} and \cite{Borel} for details.) Now, by definition, $H^\bullet_c(G;\R)$ is generated by the even and odd generators, and the image of the map $f_{G_u}^*$ contains all even generators. Hence we see from (\ref{T_Gexplicit}) that all even generators of $H^\bullet (\X_u;\R)$ are contained in the image of the geometric map $T_{G}$ and hence, by Proposition \ref{MainConvenient}, also in the image of the universal map $\sig_{G}$. They are thus bounded by Proposition~\ref{GromovConvenient}. This proves Theorem \ref{MainThm2}.

We close this subsection explaining how Theorem~\ref{MainThm} follows from Theorem~\ref{MainThm2}. By \cite[Sec.\,10]{Cartan2}, the number of odd generators of $H^\bullet (\X_u;\R)$ is given by ${\rm rk}(G_u) - {\rm rk}(K)$. Whence $f_{G_u}^*$ is onto if and only if ${\rm rk}(G_u) = {\rm rk}(K)$. The next proposition shows that this condition is always satisfied if $G$ is Hermitian.

\begin{proposition} \label{CorBorel}
	Let $G$ be a Hermitian semsimimple Lie group without compact factors and with finite center, let $K$ be its maximal compact subgroup, and let $G_u$ be its compact dual. Then ${\rm rk}(G_u) = {\rm rk}(K)$. In particular, $H^\bullet_c(G;\R)$ is generated by its even generators.
\end{proposition}

\begin{proof}
	This follows form the structure theory of Hermitian Lie groups \cite[Chapter~VII.9]{Knapp}: Namely, if $G$ is Hermitian, then its Lie algebra $\L g$ admits a compact 
Cartan subalgebra $\hfr \subset \kfr$. By definition this means that $\hfr\otimes \C \subset \gfr \otimes \C$ is a Cartan subalgebra. But since $\gfr \otimes \C = \gfr_u \otimes
\C$ this implies that $\hfr \subset \gfr_u$ is a Cartan subalgebra as well. Hence ${\rm rk}_\R(K) = \dim \hfr = {\rm rk}_\R(G_u)$.
\end{proof}

Theorem \ref{MainThm} is now an immediate consequence of Theorem \ref{MainThm2} and Proposition \ref{CorBorel}. Note that, for general semisimple Lie groups $G$, the subring of $H^\bullet_{c}(G; \R)$ generated by the even generators may be quite small. For example, for $G = SL(n, \R)$ it is generated by the Euler class and does not contain any stable classes.

\medskip

The remainder of this article is organized as follows. In Section~\ref{SecUniversal}, we show existence of the universal map $\sigma_G$ and prove Proposition \ref{GromovConvenient}. Section~\ref{SecHirzebruch} then provides a generalization of Hirzebruch's proportionality principle \cite{HirzAutom}, which will play a central role in the proof of Proposition \ref{MainConvenient} in Section \ref{SectionExplicitDescription}.

Throughout this article we will frequently make use of an auxiliary lattice $\Gamma$ in $G$. The role of this lattice will be discussed in the next section.

\section{The role of the auxiliary lattice} \label{SubsecAuxiliary}

Let $G$ be a semisimple Lie group with finite center and without compact factors. Throughout this article we shall use the following notation. We fix a maximal compact subgroup $K$ of $G$. Writing $\gfr$ and $\kfr$ for the Lie algebras of $G$ and $K$, and denoting by $\pfr$ the Killing orthogonal complement of $\kfr$ in $\gfr$, we have a Cartan decomposition $\gfr = \kfr \oplus \pfr$. This yields an identification of the tangent space $T_{o}\X$ of the symmetric space $\X = G/K$ of $G$ at the basepoint $o = eK$ with $\pfr$. By definition the compact dual group $G_u$ of $G$ is the analytic subgroup of the universal complexification $G_\C$ of $G$ with Lie algebra $\gfr_u := \kfr \oplus i\,\pfr$. Its homogeneous space $\X_u \deq G_u/K$ is called the compact dual symmetric space of $\X$. Its tangent space $T_{o_{u}}\X_{u}$ at the basepoint $o_{u} = eK$ then gets identified with $i\,\pfr$.

The assignment $\X \mapsto \X_u$ gives rise to a duality between Riemannian symmetric spaces (without Euclidean factors) of the non-compact and compact type \cite{Helgason}. However, this duality is not directly reflected in cohomology, since $\X$ is contractible and so $H^\bullet (\X;\R)$ is trivial. To overcome this problem, we fix a cocompact lattice $\Gamma$ in $G$ and consider the locally symmetric space $M := \Gamma\backslash \X$ instead of $\X$. As we shall see in Section \ref{SecHirzebruch}, the cohomology rings of $M$ and $\X_u$ do reflect the duality between $\X$ and $\X_u$. This will be made precise in the generalized Hirzebruch principle of Proposition \ref{PropositionDualityOfCharacteristicClasses}. It will be convenient to choose $\Gamma$ torsion-free so that $M$ becomes a manifold.

Note that the group cohomology of $\Gamma$ is canonically isomorphic to the singular cohomology of the manifold $M$. The next lemma shows that, by our assumptions, this cohomology ring contains the continuous cohomology ring of $G$.

\begin{lemma} \label{LemmaTransfer}
	Let $\Gamma$ be a torsion-free, cocompact lattice in $G$ and let us denote by \mbox{$\iota_\Gamma: \Gamma \hookrightarrow G$} the
inclusion. Then
\[
\lmap{\iota_\Gamma^*}{H^\bullet(G;\R)}{H^\bullet(\Gamma;\R) \cong H^\bullet(M;\R)}
\]
is injective.
\end{lemma}

\begin{proof}  A left inverse of $\iota_\Gamma^*$ is given by the transfer map \cite{BC}, which on the level of cochains is given by
\[
T^n: C(G^{n+1})^{\Gamma} \to C(G^{n+1})^{G}, \quad f \mapsto \bar f,
\]
where $\bar f$ is given in terms of the $G$-invariant probability measure $\mu$ on $G/\Gamma$ by
\[
\bar f(g_0, \dots, g_n) = \int_{G/\Gamma} f(\dot gg_0, \dots, \dot gg_n)d\mu(\dot g).
\]
\end{proof}

The lemma will later allow us to carry out computations in cohomology using concrete harmonic differential forms on $M$ rather than abstract cohomology classes. Moreover, it gives rise to the following characterization of the universal map $\sigma_G$.

\begin{corollary} \label{NaturalCharacterized}
Let $G$ be a semisimple Lie group without compact factors and with finite center and let $\Gamma$ be a torsion-free, cocompact lattice in $G$. Denote by \mbox{$\iota_\Gamma: \Gamma \hookrightarrow G$} the
inclusion. Let \mbox{$\map{\sigma}{H^\bullet(BG;\R)}{H^\bullet_c(G;\R)}$} be a homomorphism. Then $\sigma$ agrees with the universal map $\sigma_G$ if and only if the diagram
\[
\begin{xy}\xymatrix{
	H^\bullet(BG;\R) \ar[d]_{(B\iota_\Gamma)^*} \ar[r]^\sigma & H^\bullet_c(G;\R)  \ar[d]^{\iota_\Gamma^*} \\
	H^\bullet(B\Gamma;\R) \ar[r]^{\sigma_\Gamma} & H^\bullet(\Gamma;\R) \\
}\end{xy}
\]
commutes.
\end{corollary}

\begin{proof} The existence of the diagram follows from the universal property of $\sigma_G$. Conversely, the diagram determines $\sigma$ uniquely since the right down arrow is injective by Lemma \ref{LemmaTransfer}.
\end{proof}

We now fix a cocompact, torsion-free lattice $\Gamma$ in $G$ once and for all. All subsequent arguments will be independent of this choice.

\section{The image of the universal map} \label{SecUniversal}

The aim of this section is to give an explicit construction of the universal map \mbox{$\sigma_G: H^\bullet(B_*G; \R) \to H_c^\bullet(G; \R)$}, whose existence we postulated in Section~\ref{SecOverview}, and to deduce Proposition \ref{GromovConvenient} from Gromov's results on boundedness of primary characteristic classes \cite{Gromov}.

\subsection{Existence of the universal map}\label{SubsecExistenceUniversal} For an arbitrary Hausdorff locally compact topological group $G$, the \emph{continuous cohomology} $H^\bullet_c(G; \R)$ of $G$ with real coefficients is defined as the cohomology of the complex $(C^\bullet_c(G;\R), d)$, where
\[
C^n_c(G; \R) := C(G^{n+1}, \R)^G, \quad df(g_0, \dots, g_n) = \sum_{i=0}^n (-1)^i f(g_0, \dots, \widehat{g_i}, \dots, g_n).
\]
Here $C(\cdot, \R)$ stands for real-valued continuous functions and $(-)^G$ denotes the functor of $G$-invariants. Dropping the continuity requirement on the cochains we obtain the usual group cohomology $H^\bullet(G; \R)$. Thus, if $G^\delta$ denotes $G$ equipped with the discrete
topology, then by definition 
\[
H^\bullet(G; \R) = H^\bullet_c(G^\delta; \R).
\]
An important difference between continuous cohomology and ordinary group cohomology is that the latter has a direct geometric interpretation. Namely, if $BG^{\delta}$ denotes any
classifying space for $G^\delta$, then $H^\bullet(G; \R) \cong H^\bullet(BG^{\delta}; \R)$. However, it is in general \emph{not} true that $H^\bullet_c(G; \R) \cong H^\bullet(BG;
\R)$.

There are various ways to understand the difference between the groups $H^\bullet_c(G; \R)$ and $H^\bullet(BG;\R)$. The classical point of view is to consider $H^\bullet_c(G; \R)$ as
relative cohomology groups on a space with two topologies; this was pioneered by Bott \cite{Bott} and developed further by M. A. Mostow \cite{Mostow}. We will take a different
point of view here, which is inspired from groupoid cohomology \cite{Tu}. Namely, given a Hausdorff topological group $G$ we define a topological simplicial object
$G_\bullet$ by $G_n := G^{n}$, $n \geq 0$ with the usual face and degeneracy maps as described in \cite[p.\,76]{DupontLNM} (where $G_\bullet$ is denoted $NG(\bullet)$). In
order to avoid technicalities we will assume that $G$ is either a Lie group or a (not necessarily countable) discrete group. Then
the fat geometric realization $B_*G := \|G_\bullet\|$ of $G_\bullet$ is a model for the classifying space of $G$, and this model is functorial. In fact, it coincides with the
Milnor model \cite{Milnor}. On each of the
spaces $G_n$ we can consider the sheaf $\underline{\R}$ of locally constant real valued functions and the sheaf $C^0$ of continuous real valued functions. Since these sheaves are
compatible with the face and degeneracy maps, we obtain sheaves $\underline{\R}$ and $C^0$ over the simplicial space $G_\bullet$. (For the notion of a sheaf over a simplicial
space see \cite[Sec.\,3]{Tu}.) We can thus form the corresponding sheaf cohomology groups. Now the sheaf $\underline{\R}$ admits a flabby resolution 
\[\underline{\R} \to C^0 \xto{d} C^1 \to \dots \]
by the sheaves $C^q$ of singular real $q$-cochains (i.e. for $U_n \subset G_n$ the group $C^q(U_n)$ consists of singular real $q$-cochains in $U_n$.) Hence $H^\bullet(G_\bullet;
\underline{\R})$ is the cohomology of the total complex associated to the double complex $\{C^q(G_n)\}$. Then \cite[Prop.\,5.15]{DupontLNM} applies and we obtain
\[H^\bullet(G_\bullet; \underline{\R}) \cong H^\bullet(\|G_\bullet\|; \underline{\R}) \cong H^\bullet(B_*G; \R).\]
On the other hand, the sheaves $C^0$ on $G_n$ are flabby, hence acyclic, and thus the double complex computing $H^\bullet(G_\bullet; C^0)$ collapses to the inhomogeneous bar resolution for $
H_c^\bullet(G; \R)$. This implies
\[H^\bullet(G_\bullet; C^0) \cong H_c^\bullet(G; \R).\]
In particular, the inclusion $i: \underline{\R} \to C^0$ of sheaves induces a map
\begin{eqnarray}\label{DefSigma}\sigma_G: H^\bullet(B_*G; \R) \cong H^\bullet(G_\bullet; \underline{\R}) \xrightarrow{i^*} H^\bullet(G_\bullet; C^0) \cong H_c^\bullet(G; \R).\end{eqnarray}
We can think of this map as follows. The complex which computes $H_c^\bullet(G; \R) \cong H^\bullet(G_\bullet; C^0)$ is the $0$-th row of the double complex $\{C^q(G_n)\}$ which computes
$H^\bullet(B_*G; \R)$. The map $\sigma_G$ is then induced by the projection maps
\[\bigoplus_{p+q = n} C^q(G_p) \to C^0(G_n).\]
Let us spell out the naturality property of $\sigma_G$ explicitly. \mbox{$G \mapsto B_*G \mapsto H^\bullet(B_*G; \R)$} and $G \mapsto H_c^\bullet(G; \R)$ define two functors on the category formed by Lie groups and discrete groups with the obvious
morphisms, and $\sigma_G$ is
a natural transformation between these functors. Moreover, if $\Gamma$ is discrete, then every continuous function is locally constant, whence the sheaves $\underline{\R}$ and
$C^0$ on $\Gamma_\bullet$ coincide, so that $\sigma_\Gamma$ is an isomorphism in this case.

We have thus shown that $\sigma_G$ has the universal property used to define the universal map in Section \ref{SecOverview}. By Corollary \ref{NaturalCharacterized} this property characterizes $\sigma_G$ uniquely, so that we may indeed refer to $\sigma_G$ as \emph{the} universal map from $H^\bullet(B_*G; \R)$ to $H_c^\bullet(G; \R)$.

\subsection{Gromov's boundedness theorem} \label{SubsecGromov}

We claimed in Section \ref{SecOverview} that Proposition~\ref{GromovConvenient} is a direct consequence of Gromov's results on boundedness of primary characteristic classes \cite{Gromov}. We shall now make this precise.

We begin by recalling the necessary background on bounded cohomology. For an arbitrary locally compact Hausdorff topological group $G$, the continuous bounded cohomology ring
$H_{cb}^\bullet(G;\R)$ of $G$ is defined as be the cohomology of the complex $(C^\bullet_{cb}(G;\R), d)$ of continuous 
bounded functions, where
\[
C^n_{cb}(G, \R) := C_b(G^{n+1}, \R)^G, \quad df(g_0, \dots, g_n) = \sum_{i=0}^n (-1)^i f(g_0, \dots, \widehat{g_i}, \dots, g_n).
\]
If $\Gamma$ is a discrete group, we drop the subscript $c$ from notation, writing $H^\bullet(\Gamma;\R) $ and $H_{b}^\bullet(\Gamma;\R)$. In this case, the groups $H_b^\bullet(\Gamma; \R)$ are precisely Gromov's bounded cohomology groups from \cite{Gromov}. The generalization to topological groups is due to Burger and Monod \cite{BuMo}. Its basic properties are summarized in \cite{BC}, from which we recall the following facts. Firstly, the inclusion of complexes
\[(C^\bullet_{cb}(G;\R), d) \hookrightarrow (C^\bullet_{c}(G;\R), d) \]
induces a \emph{comparison map}
\[
\lmap{c^\bullet_G}{H_{cb}^\bullet(G;\R)}{H_c^\bullet(G;\R)}
\]
which is natural in $G$. Secondly, if $H \subset G$ is a subgroup of finite covolume then there exists a \emph{bounded transfer map}
\[
\lmap{T_b^\bullet}{H_{cb}^\bullet(H;\R)}{H_{cb}^\bullet(G;\R)}
\]
which provides a left-inverse to the restriction map $H_{cb}^\bullet(G;\R) \to H_{cb}^\bullet(H;\R)$, and on the level of cochains is given by the same formula as the usual transfer map defined in the proof of Lemma \ref{LemmaTransfer} above.

We now return to the case where $G$ is a semisimple Lie group without compact factors and with finite center, and
denote by $G^\delta$ the discrete group underlying $G$. We denote by $B_*$ a functorial model of the classifying space (say,
the Milnor model) and consider the map $B_*\iota_{G^\delta}: B_*G^\delta \to B_*G$ induced by the continuous map $\iota_{G^\delta}: G^\delta \to G$. The elements in the image of the map
\[
H^\bullet(B_*G) \xrightarrow{(B_*\iota_{G^\delta})^*} H^\bullet(B_*G^\delta) \cong  H^\bullet(G^\delta; \R)
\]
are called \emph{primary characteristic classes}. The following result of Gromov ensures that primary characteristic classes are bounded.
\begin{auxthm}[Gromov \cite{Gromov}] \label{Gromov}
	Every primary characteristic class lies in the image of the comparison map $c_{G^\delta}^\bullet: H_b^\bullet(G^\delta;\R) \to H^\bullet(G^\delta;\R)$.
\end{auxthm}
For an alternative approach to Theorem \ref{Gromov} we refer the reader to \cite{Michelle}, where a stronger version of the theorem is proved.

We are now ready to prove Proposition \ref{GromovConvenient}. We shall be using the (torsion-free, cocompact) auxiliary lattice $\Gamma$ in $G$ introdcued in Section \ref{SubsecAuxiliary}.

The inclusion $\iota_\Gamma: \Gamma \hookrightarrow G$ factors as $\iota_\Gamma = \iota_\delta \circ \iota_\Gamma^\delta$, where $\iota_\Gamma^\delta: \Gamma \to G^\delta$ and $\iota_\delta: G^\delta \to G$. Then the universal property of the universal map $\sigma_G$ yields the following commutative diagram:
\begin{eqnarray*}
\begin{xy}
  \xymatrix{
      & H_{c}^{\bullet}(G; \R)  \ar[d]_{\iota_{\delta}^*} &   \ar_{\sigma_G}[l] \ar[d]^{(B_*\iota_{\delta})^*} H^{\bullet}(B_*G; \R) \\
      H^{\bullet}_b(G^\delta; \R) \ar[r]^{c_{G^\delta}^\bullet} \ar[d]_{(\iota_\Gamma^\delta)^*} & H^{\bullet}(G^\delta; \R) \ar[d]_{(\iota_\Gamma^\delta)^*}&    \ar[l]_{\sigma_{G^\delta}}  H^{\bullet}(B_*G^\delta; \R) \\
      H^{\bullet}_b(\Gamma; \R) \ar[r]^{c_G^\bullet} \ar[d]_{T_b^{\bullet}} & H^{\bullet}(\Gamma; \R) \ar[d]_{T^\bullet}& \\
      H_{cb}^{\bullet}(G; \R) \ar[r]^{c_G^\bullet} & H_{c}^{\bullet}(G; \R)&  \\
  }
\end{xy}
\end{eqnarray*}
Thus if $\alpha$ lies in the image of the universal map $\sig_{G}$, then $\iota_{\delta}^*\alpha$ is a primary characteristic class. By Theorem \ref{Gromov} we thus find $\beta \in H^{\bullet}_b(G^\delta; \R)$ with $\iota_{\delta}^*\alpha = c^{\bullet}_{G^\delta}(\beta)$. Then commutativity of the diagram yields
\[
\alpha = T^\bullet(\iota_{\Gamma}^*\alpha) = T^\bullet \circ \bigl( \iota_\Gamma^\delta \bigr)^* \bigl( \iota_{\delta}^*\alpha \bigr) = T^\bullet \circ \bigl( \iota_\Gamma^\delta \bigr)^* \bigl( c^{\bullet}_{G^\delta}(\beta)  \bigr) = c_G^{\bullet} \bigl( T_b^{\bullet}\circ\iota_{\delta}^*(\beta) \bigr),
\]
that is, $\al$ lies in the image of the comparison map $c^{\bullet}_{G}$.

\section{Duality and cohomology of symmetric spaces}\label{SecHirzebruch}

In this section we recall the van Est isomorphisms, which relates the continuous cohomology of a semisimple Lie group $G$ to the singular cohomology of its compact dual symmetric space $\X_u$. The latter is then related to the cohomology of the locally symmetric space $M$ of the fixed auxiliary lattice $\Gamma$ in $G$ by means of a generalized Hirzebruch proportionality principle. We emphasize that while the classical Hirzebruch proportionality principle \cite{HirzAutom} and its generalizations \cite{KamberTondeur} relate characteristic numbers of bundles over $\X_u$ to characteristic numbers of bundles over $M$, our generalization provides a relation for characteristic classes of arbitrary degree.

\subsection{The van Est isomorphism}\label{SubsecVanEst}\label{SubSectionDualityOfSymmetricSpaces}

Throughout this section, $G$ denotes a semisimple Lie group without compact factors and with finite center. Moreover, we denote by $\iota_\Gamma: \Gamma \hookrightarrow G$ the inclusion of the fixed auxiliary lattice $\Gamma$, and by $M := \Gamma\backslash\X$ the associated locally symmetric space. We will again be considering the groups $H_c^n(G; \R)$, but from a point of view different from the one of the last section. Namely, we will take the point of view of  \cite{BoWa}, thinking of $H_c^n(G; \R)$ as derived functors rather than as sheaf cohomology groups of a simplicial manifold. Accordingly, we can use any $s$-injective resolution
\begin{eqnarray} \label{GRes}
	0 \to \R \to A^0 \xto{d_0} A_1 \to \dots
\end{eqnarray}
in order to compute $H_c^\bullet(G; \R)$ as the cohomology of the complex
\begin{eqnarray} \label{InvComplex}
	(A^0)^G \xto{d_0} (A_1)^G \to \dots
\end{eqnarray}
of $G$-invariants (see \cite[p.\,261]{BoWa}). In fact, there exists a continuous chain map, unique up to $G$-chain homotopy, from the complex \eqref{GRes} into the augmented homogeneous bar resolution, which identifies the cohomology of \eqref{InvComplex} with $H_c^\bullet(G; \R)$. The complex \eqref{GRes} is also $\Gamma$-injective, whence there exists an
isomorphism between the cohomology of the complex
\begin{eqnarray}\label{InvComplexGamma}(A^0)^\Ga \xto{d_0} (A_1)^\Ga \to \dots\end{eqnarray}
and $H^\bullet(\Gamma; \R)$. Via these isomorphisms the map $\iota_\Gamma^*: H_c^\bullet(G; \R) \to H^\bullet(\Gamma; \R)$ is intertwined with the inclusion map $(A_n)^G \hookrightarrow (A_n)^\Ga$.

Now, for the symmetric space $\X$ of $G$
\[
0 \to \R \to \Omega^0(\X) \xto{\dop} \Omega^1(\X) \to \dots
\]
is an $s$-injective resolution \cite[Prop. 5.4]{BoWa}. Using the fact that $G$-invariant forms on $\X$ are harmonic, we thus obtain an isomorphism
\[
\lmap{\iota^\bullet_{vE}}{\Omega^\bullet(\X)^G \cong H^\bullet(\Omega^\bullet(\X)^G, \dop)}{H_c^\bullet(G; \R)},
\]
called the \emph{van Est isomorphism} (see \cite{DupontTopology} for an explicit description of this map on the level of cochains).

A more algebraic model of the space $\Omega^\bullet(\X)^G$ can be obtained as follows. Since the action of $G$ on $\X$ is transitive, any $G$-invariant differential form $\omega
\in \Omega^\bullet(\X)^G$ is uniquely determined by its value $\omega_o$ at the base point. We thus have an isomorphism
\[\Omega^\bullet(\X)^G \cong \left(\bigwedge{}^\bullet \pfr^*\right)^K, \quad \omega \mapsto \omega_o.\]
The right hand side actually is the $(\gfr, K)$-cohomology $H^\bullet(\gfr, K; \R)$ with trivial coefficients. This $(\gfr, K)$-cohomology has the following interpretation in terms of the compact dual symmetric space of $G$: Identify the singular cohomology groups of $\X_u$ with the de Rham cohomology groups $H^{\bullet}\bigl( \Omega^{\bullet}(\X_{u}), \dop^{\bullet} \bigr)$. Now any differential form on $\X_{u}$ can be made
$G_u$-invariant by integration, and this integration does not affect the cohomology classes of closed forms. Moreover, $G_u$-invariant forms are automatically harmonic. Thus Hodge
theory provides an isomorphism
\begin{eqnarray} \label{MapIsosHodge}
	H^{\bullet}\bigl( \Omega^{\bullet}(\X_{u}), \dop^{\bullet} \bigr) \cong H^{\bullet}\bigl( \Omega^{\bullet}(\X_{u})^{G_{u}}, \dop^{\bullet} \bigr) \cong \Omega^{\bullet}(\X_u)^{G_u},
\end{eqnarray}
where $\Omega^{\bullet}(\X_{u})^{G_{u}}$ denotes the subcomplex of $G_{u}$-invariant forms. Furthermore, restricting $G_{u}$-invariant forms on $\X_{u}$ to the basepoint $o_{u}$ we obtain an isomorphism
\begin{eqnarray} \label{XuAlgebraic}
	\Omega^{\bullet}(\X_{u})^{G_u} \cong \left( \bigwedge{}^{\bullet}\,T_{o_{u}}^{\ast}\X_{u} \right)^{K} \cong \left( \bigwedge{}^\bullet (i\,\pfr)^{\ast} \right)^{K}, \jump \om \mapsto \om_{o_{u}}
\end{eqnarray}
which identifies $G_{u}$-invariant forms on $\X_{u}$ with $K$-invariant anti-symmetric multilinear forms on $i\,\pfr$. The linear map $\map{\io}{\pfr}{i\,\pfr}$, $X \mapsto iX$ then induces an isomorphism
\begin{eqnarray} \label{MapIsoiota}
	\lmapcong{\io^{\ast}}{\left( \bigwedge{}^{\bullet} (i\,\pfr)^{\ast} \right)^{K}}{\left( \bigwedge{}^{\bullet} \pfr^{\ast} \right)^{K}}
\end{eqnarray}
which acts on n-linear forms by
\[
(\io^{\ast}\,\al)(X_{1},\dots,X_{n}) = \al(iX_{1},\dots,iX_{n}), \jump X_{1},\ldots,X_{n} \in \pfr.
\]
We have thus obtained an isomorphism
\[H^\bullet(\X_u; \R) \cong \left( \bigwedge{}^{\bullet} \pfr^{\ast} \right)^{K}\]
realizing the $(\L g, K)$-cohomology groups in question as singular cohomology groups of $\X_u$. In particular we obtain isomorphisms
\begin{eqnarray} \label{CompactModel}
	\lmap{\Phi_G}{H^\bullet(\X_u;\R)}{\Omega^\bullet(\X)^G}
\end{eqnarray}
and
\begin{eqnarray} \label{CompactModel2}
	\lmap{\Psi_G := \iota_{vE}^\bullet \circ \Phi_G}{H^\bullet(\X_u;\R)}{H_c^\bullet(G;\R)}.
\end{eqnarray}
Here the isomorphism $\Phi_G$ is completely explicit on cochains. We will now relate these isomorphisms to the auxiliary lattice $\Gamma$ and the associated locally symmetric space $M = \Gamma\backslash \X$. Recall at this point that $\Gamma$ is assumed to be cocompact and torsion free. We denote by $\map{\pi}{\X}{M}$ the canonical projection and observe that the image of the inclusion $\Om^{\bullet}(\X)^{G} \hookrightarrow \Om^{\bullet}_{\rm{cl}}(\X)^{\Ga} \cong \Omega^\bullet(M)$ consist of closed forms. In particular, we obtain a push-forward map
\begin{eqnarray} \label{MapFromXToM}
	\lmap{\pi_{!}}{\Om^{\bullet}(\X)^{G}}{H^{\bullet}(M;\R)}.
\end{eqnarray}
Since the inclusion $\map{\io_{\Gamma}}{\Gamma}{G}$ induces the restriction map \mbox{$\map{\iota_\Gamma^*}{H^\bullet_c(G;\R)}{H^\bullet(\Gamma;\R)}$},  we have thus proved the following proposition.

\begin{proposition} \label{SmoothImplementation}
	If $i_M:  H^\bullet_{dR}(M; \R)  \to H^\bullet(\Gamma; \R)$ denotes the canonical isomorphism, then the diagram
\[\begin{xy}\xymatrix{
\Omega(\X)^G \ar[d]_{\iota_{vE}} \ar[r]^{\pi_!}& H^\bullet_{dR}(M; \R) \ar[d]^{i_M}\\
H^\bullet_c(G; \R) \ar[r]^{\iota_\Gamma^*}& H^\bullet(\Gamma; \R)
}\end{xy}\]
commutes.
\end{proposition}

The generalized Hirzebruch proportionality principle will be formulated in terms of the homomorphism $\map{\Phi_\Gamma}{H^\bullet(\X_u;\R)}{H^\bullet(M; \R)}$ defined by 
\begin{eqnarray} \label{MapBetweenCohomologyOfSymmetricSpaces}
	\lmap{\Phi_{\Ga} := \pi_! \circ \Phi_G}{H^{\bullet}(\X_{u};\R)}{H^{\bullet}(M;\R)}.
\end{eqnarray}
It will thus be convenient to have the following reformulation of Proposition \ref{SmoothImplementation}.

\begin{corollary} \label{LatticeInclusion}
	The diagram
\[
\begin{xy}\xymatrix{
	H^\bullet(\X_u;\R) \ar[r]^{\Phi_\Gamma} \ar[d]_{\Psi_G} & H^\bullet_{dR}(M;\R) \ar[d]^{i_M}\\
	H_{c}^\bullet(G;\R) \ar[r]^{\iota_\Gamma^*} & H^\bullet(\Gamma;\R).}
\end{xy}
\]
commutes.
\end{corollary}

\subsection{Characteristic classes of principal bundles}
\label{SubSectionCharacteristicClasses}

In this subsection we collect some basic facts from Chern-Weil theory, that is, the theory of characteristic classes of principal bundles over smooth manifolds. We follow the exposition in \cite{DupontLNM}.

Let $BK$ denote a classifying space for the compact Lie group $K$. The elements of $H^{\bullet}(BK;\R)$ are characteristic classes of principal $K$-bundles. More specifically, given a smooth manifold $X$ and a principal $K$-bundle $P \to X$ with classifying map $\map{f_{P}}{X}{BK}$, any characteristic class $c \in H^{\bullet}(BK;\R)$ gives rise to a corresponding characteristic class
\[
c(P) \deq f_{P}^{\ast}\,c \in H^{\bullet}(X;\R)
\]
of the bundle $P$. For any connection 1-form $A \in \Om^{1}(P,\kfr)$ on $P$, its curvature 2-form $F_{A} \in \Om^{2}(P,\kfr)$ is given by
\[
F_{A} = \dop\!A + \frac{1}{2} [A \wedge A].
\]
Here $[A \wedge A]$ denotes the 2-form defined by
\[
[A \wedge A](v,w) = [A(v),A(w)] - [A(w),A(v)] = 2\,[A(v),A(w)]
\]
for tangent vectors $v,w$ on $P$. The curvature form $F_{A}$ is horizontal and thus descends to a 2-form
\[
F_{A} \in \Om^{2}\bigl( X,P(\kfr) \bigr)
\]
on $X$ with values in the adjoint bundle $P(\kfr) := P \xop_{K} \kfr$. Recall that $K$ acts on the space $S^{k}(\kfr^{\ast})$ of symmetric $k$-multilinear functions on $\kfr$ via the diagonal adjoint action. We denote by $I^{k}(\kfr^{\ast}) \subset S^{k}(\kfr^{\ast})$ the subset of those functions that are invariant under this action. Given a $K$-invariant symmetric function $f \in I^{k}(\kfr^{\ast})$ and a connection 1-form $A \in \Om^{1}(P,\kfr)$, we obtain a well-defined closed 2-form
\[
f(F_{A},\cdots,F_{A}) \in \Om^{2}(X)
\]
on $X$ which then defines a class in $H^{2}_{\rm{dR}}(X)$. We now have the following lemma (see Theorem 8.1 in \cite{DupontLNM}):

\begin{lemma} \label{LemmaChernWeil}
	Let $K$ be a compact Lie group.
\begin{enumerate} \itemsep 1ex
	\item There are no characteristic classes of principal $K$-bundles in odd degree, that is,
\[
H^{2k+1}(BK; \R) = \{0\}
\]
for all $k \geq 0$.
	\item For every characteristic class $c \in H^{2k}(BK;\R)$ there exists a unique $f \in I^{k}(\kfr^{\ast})$ such that the following holds.

For every principal $K$-bundle $P \to X$ over a compact smooth manifold $X$ and every connection 1-form $A \in \Omega^1(P,\kfr)$ with curvature form $F_{A} \in \Omega^{2}(X;P(\kfr))$, the corresponding characteristic class $c(P) \in H^{2}(X;\R)$ of the bundle $P$ gets identified with the class
\[
[f(F_{A},\cdots,F_{A})] \in H^{2}_{\rm{dR}}(X)
\]
under the de Rham isomorphism.
\end{enumerate}
\end{lemma}

\subsection{Duality of characteristic classes}
\label{SubSectionDualityOfCharacteristicClasses}

The symmetric spaces $\X$ and $\X_{u}$ come along with canonical principal $K$-bundles
\[
G \lto \X=G/K, \jump G_{u} \lto \X_{u}=G_{u}/K.
\]
The former bundle further descends to a principal $K$-bundle
\[
\Ga \backslash G \lto M=\Ga \backslash \X
\]
over the locally symmetric space $M$. As we have seen in Section \ref{SubSectionCharacteristicClasses} above, any characteristic class $c \in H^{2k}(BK;\R)$ yields corresponding characteristic classes
\[
c(\Ga \backslash G) \in H^{\bullet}(M;\R) \jump \text{and} \jump c(G_{u}) \in H^{\bullet}(\X_{u};\R)
\]
of the canonical $K$-bundles $\Ga\backslash G \to M$ and $G_{u} \to \X_{u}$, respectively. The following proposition gives an explicit relation between these classes.

\begin{proposition} \label{PropositionDualityOfCharacteristicClasses}
	Let $k \ge 0$, and let $c \in H^{2k}(BK;\R)$ be a characteristic class. Then the corresponding characteristic classes of the canonical $K$-bundles $G_{u} \to \X_{u}$ and $\Ga\backslash G \to M$ are related by
\[
\Phi_{\Ga}\bigl( c(G_{u}) \bigr) = (-1)^{k} \cdot c(\Ga \backslash G),
\]
where $\map{\Phi_{\Ga}}{H^{2k}(\X_{u};\R)}{H^{2k}(M;\R)}$ is the homomorphism (\ref{MapBetweenCohomologyOfSymmetricSpaces}).
\end{proposition}

\begin{proof}
	We will prove the claimed relation by unraveling the definition of the homomorphism $\Phi_{\Ga}$ given in Section \ref{SubSectionDualityOfSymmetricSpaces} above, using the results from Chern-Weil theory discussed in the previous subsection.

First of all, we introduce appropriate connection 1-forms on the bundles $\Ga\backslash G \to M$ and $G_u \to \X_u$. Let us denote by $\th_{G} \in \Om^{1}(G;\gfr)$ and $\th_{G_{u}} \in \Om^{1}(G_{u};\gfr_{u})$ the left Maurer-Cartan forms on $G$ and $G_{u}$ respectively, and denote by
\[
\lmap{\pi_{\kfr}}{\gfr = \kfr \oplus \pfr}{\kfr}, \jump \lmap{\pi_{\kfr}^{u}}{\gfr_{u} = \kfr \oplus i\,\pfr}{\kfr}
\]
the canonical projections. Then
\[
\tilde{A} \deq \pi_{\kfr} \circ \th_{G} \in \Om^{1}(G;\kfr)
\]
defines a connection 1-form on the bundle $G \to \X$. This form is invariant under the action of the auxiliary lattice $\Ga$ in $G$ and hence descends to a connection 1-form $A \in \Om^{1}(\Ga \backslash G;\kfr)$ on the bundle $\Ga \backslash G \to M$. Likewise,
\[
A_{u} \deq \pi_{\kfr}^{u} \circ \th_{G_{u}} \in \Om^{1}(G_{u};\kfr)
\]
defines a connection 1-form on the bundle $G_{u} \to \X_{u}$. \\

Now we are ready for the actual proof of the proposition: Let $c \in H^{2k}(BK; \R)$ be a characteristic class and denote by $f \in I^{k}(\kfr^{\ast})$ the corresponding $k$-multilinear invariant function on $\kfr$ as in Lemma \ref{LemmaChernWeil}.
By Lemma \ref{LemmaChernWeil}, the characteristic class
\[
c(G_{u}) \in H^{2k}(\X_{u};\R)
\]
is represented as a class in de Rham cohomology by the closed form
\[
f(F_{A_{u}},\dots,F_{A_{u}}) \in \Omega^{2k}(\X_{u}).
\]
Since the Maurer-Cartan form $\th_{G_{u}} \in \Om^{1}(G_{u};\kfr)$ is $G_{u}$-inavariant, it follows that the connection 1-form $A_{u} \in\Om^{1}(G_{u};\kfr)$ and hence also the form $f(F_{A_{u}},\dots,F_{A_{u}}) \in\Om^{2k}(\X_{u})$ on $\X_{u}$ is $G_{u}$-invariant. Thus the class $[f(F_{A_{u}},\dots,F_{A_{u}})]$ is mapped under (\ref{MapIsosHodge}) to the invariant form
\[
f(F_{A_{u}},\dots,F_{A_{u}}) \in\Om^{2k}(\X_{u})^{G_{u}}.
\]
Restricting this form to the basepoint $o_{u}$ we see that it gets mapped further under (\ref{XuAlgebraic}) to the $K$-invariant $2k$-linear form
\[
\bigl( f(F_{A_{u}},\dots,F_{A_{u}}) \bigr)_{o_{u}} \in \left( \bigwedge{}^{2k} (i\,\pfr)^{\ast} \right)^{K}
\]
on $i\,\pfr$. In order to figure out the image of this form under the isomorphism (\ref{MapIsoiota}) we use the following fact.

\begin{claim*}
\[
\io^{\ast} \Bigl( \bigl( f(F_{A_{u}},\dots,F_{A_{u}}) \bigr)_{o_{u}} \Bigr) = (-1)^{k} \cdot \bigl( f(F_{\tilde{A}},\dots,F_{\tilde{A}}) \bigr)_{o}.
\]
\end{claim*}

We will prove the claim by a direct calculation. First, we note that it follows from the definitions in a straightforward way that, for $iX_{1},\ldots,iX_{2k}\in i\,\pfr \cong T_{o_{u}}\X_{u}$, we have
\begin{multline*}
(f(F_{A_{u}},\dots,F_{A_{u}}) \bigr)_{o_{u}}(iX_{1},\ldots,iX_{2k}) \\
	= \frac{1}{(2k)!} \sum_{\sigma \in \mathfrak S_{2k}} (-1)^\sigma f((F_{A_{u}})_{e}(iX_{\sigma(1)},iX_{\sigma(2)}),\dots,(F_{A_{u}})_{e}(iX_{\sigma(2k-1)},iX_{\sigma(2k)})),	
\end{multline*}
where $e\in G_{u}$ is the unit element. On the left hand side of this identity we regard $F_{A_{u}}$ as a form on $\X_{u}$ whereas on the right hand side we regard it as a form on $G_{u}$. Likewise, for $X_{1},\ldots,X_{2k}\in \pfr \cong T_{o}\X$ we have
\begin{multline*}
(f(F_{\tilde{A}},\dots,F_{\tilde{A}}))_{o}(X_{1},\ldots,X_{2k})\\ = \frac{1}{(2k)!} \sum_{\sigma \in \mathfrak S_{2k}} (-1)^\sigma f((F_{\tilde{A}})_{e}(X_{\sigma(1)},X_{\sigma(2)}),\dots, (F_{\tilde{A}})_{e}(X_{\sigma(2k-1)},X_{\sigma(2k)})),
\end{multline*}
where $e\in G$ is the unit element of $G$. Again, on the left hand side of this identity $F_{\tilde{A}}$ is regarded as a form on $\X$ whereas on the right hand side it is considered as a form on $G$. Since $f$ is $k$-linear we see from this and the definition of the isomorphism (\ref{MapIsoiota}) that it will be enough to establish the relation
\[
\bigl( F_{A_{u}} \bigr)_{e}(iX,iY) = - \bigl( F_{\tilde{A}} \bigr)_{e}(X,Y)
\]
for $X,Y\in\pfr$. To this end, we recall that the Maurer-Cartan form on $G_{u}$ satisfies the identities
\[
(\th_{G_{u}})_{o_{u}}(iX) = iX, \, X \in \pfr \jump \text{and} \jump \dop\!\th_{G_{u}} + \frac{1}{2} \bigl[ \th_{G_{u}} \wedge \th_{G_{u}} \bigr] = 0.
\]
Then we obtain
\[
\begin{split}
	\bigl( F_{A_{u}} \bigr)_{e}(iX, iY) &= \left( \bigl( dA_{G_{u}} \bigr)_{e} + \frac{1}{2} \bigl[ A_{G_{u}} \wedge A_{G_{u}} \bigr]_{e} \right) (iX,iY) \\
&= \left( \pi_{\kfr}^{u} \circ \bigl( \dop\!\th_{G_{u}} \bigr)_{e} + \frac{1}{2} \bigl[ \pi_{\kfr}^{u} \circ \dop\!\th_{G_{u}} \wedge \pi_{\kfr}^{u} \circ \dop\!\th_{G_{u}} \bigr]_{e} \right)(iX,iY) \\
&= \frac{1}{2} \left( -\pi_{\kfr}^{u} \circ \bigl[ \th_{G_{u}} \wedge \th_{G_{u}} \bigr]_{e} + \bigl[ \pi_{\kfr}^{u} \circ \dop\!\th_{G_{u}} \wedge \pi_{\kfr}^{u} \circ \dop\!\th_{G_{u}} \bigr]_{e} \right)(iX,iY) \\
&= - \frac{1}{2} \bigl[ \th_{G_{u}} \wedge \th_{G_{u}} \bigr]_{e} (iX,iY) \\
&= - \frac{1}{2} \Bigl( \bigl[ (\th_{G_{u}})_{e}(iX),(\th_{G_{u}})_{e}(iY) \bigr] - \bigl[ (\th_{G_{u}})_{e}(iY),(\th_{G_{u}})_{e}(iX) \bigr] \Bigr) \\
&= - [iX,iY] \\
&= [X,Y].
\end{split}
\]
A similar computation shows that
\[
\bigl( F_{\tilde{A}} \bigr)_{e}(X,Y) = -[X,Y].
\]
This proves the claim.

Continuing with the proof of the proposition, we note that since the Maurer-Cartan form $\th_{G} \in \Om^{1}(G;\kfr)$ is $G$-invariant, it follows that the connection 1-form $\tilde{A} \in\Om^{1}(G;\kfr)$ and hence also the form $f(F_{\tilde{A}},\dots,F_{\tilde{A}}) \in\Om^{2k}(\X)$ on $\X$ is $G$-invariant. Hence the $2k$-linear form $\bigl( f(F_{\tilde{A}},\dots,F_{\tilde{A}}) \bigr)_{o}$ on $\pfr$ is mapped to the form
\[
f(F_{\tilde{A}},\dots,F_{\tilde{A}}) \in\Om^{2k}(\X)^{G}.
\]
This form is in particular $\Ga$-invariant, so it follows from Lemma \ref{LemmaChernWeil} that it gets mapped under (\ref{MapFromXToM}) to the characteristic class $c(\Ga\backslash G)$. The proposition is proved.
\end{proof}

\subsection{A proportionality principle}
\label{SubSectionAProportionalityPrinciple}

In order to clarify the relation between our results and the generalized Hirzebruch proportionality principle obtained in \cite[Sec.\,4.14]{KamberTondeur}, we derive a corollary of Proposition \ref{PropositionDualityOfCharacteristicClasses} concerning characteristic numbers. For any compact oriented manifold $X$ of dimension $m$, we denote by
\[
\lmap{\l\,\cdot\,,[X]\r}{H^{m}(X;\R)}{\R}
\]
the pairing of classes of top degree in cohomology with the fundamental class $[X] \in H_{m}(X;\R)$ in the singular homology of $X$. Then we have the following proportionality principle.

\begin{corollary}
	Write $m \deq \dim(M) = \dim(\X_{u})$, and fix an orientation of $\X_{u}$ and $M$. Then there exists a real number $a(\Gamma) \neq 0$ such that for any collection $c_{1},\dots,c_{r} \in H^{\bullet}(BK;\R)$ of characteristic classes satisfying
\[
\deg(c_{1}) + \dots + \deg(c_{r}) = m,
\]
we have
\[
\big\l c_{1}(G_{u}) \wedge \dots \wedge c_{r}(G_{u}),[\X_{u}] \big\r = a(\Gamma) \cdot \big\l c_{1}(\Ga \backslash G) \wedge \dots \wedge c_{r}(\Ga \backslash G),[M] \big\r.
\]
\end{corollary}

\begin{proof}
	By the assumption and Lemma \ref{LemmaChernWeil}\,(i) we may without loss of generality assume $m$ to be even. Since $H^{m}(\X_{u};\R) \cong \R$ there exists a real number $a^{\prime}$ such that the linear functionals
\[
\lmap{\big\l\,\cdot\,,[\X_{u}]\big\r,\big\l\Phi_{\Ga}(\,\cdot\,),[M]\big\r}{H^{m}(\X_{u};\R)}{\R}
\]
are related by
\[
\big\l\,\cdot\,,[\X_{u}]\big\r = a^{\prime} \cdot \big\l\Phi_{\Ga}(\,\cdot\,),[M]\big\r.
\]
Then Proposition \ref{PropositionDualityOfCharacteristicClasses} yields
\[
\begin{split}
	\big\l c_{1}(G_{u}) \wedge \dots \wedge c_{r}(G_{u}),[\X_{u}] \big\r &= a^{\prime} \cdot \big\l \Phi_{\Ga}\bigl( c_{1}(G_{u}) \wedge \dots \wedge c_{r}(G_{u}) \bigr),[M] \bigr) \big\r \\
	&= (-1)^{m/2} \cdot a^{\prime} \cdot \big\l c_{1}(\Ga \backslash G) \wedge \dots \wedge c_{r}(\Ga \backslash G),[M] \big\r.
\end{split}
\]
This shows in particular that $a^{\prime} \neq 0$. Now define $a(\Gamma) \deq (-1)^{m/2} \cdot a^{\prime}$.
\end{proof}

\section{Geometric vs. universal map}\label{SectionExplicitDescription}

The goal of this section is to identify the universal map $\sig_{G}$ and the geometric map $T_{G}$ up to sign. This will prove Proposition \ref{MainConvenient}, thereby completing the proof of Theorem \ref{MainThm2}.

\subsection{Reformulation of the statement}\label{SubSectionTransferMap}\label{SubSectionResult}

Throughout this section, $G$ denotes a semisimple Lie group without compact factors and with finite center. We keep the notation introduced in the last section. In particular, given a subgroup $H < G$ we denote by $\iota_H: H \hookrightarrow G$ the corresponding inclusion map.

Since $K$ is a maximal compact subgroup in $G$, the inclusion $\iota_K: K \hookrightarrow G$ is a homotopy equivalence and thus induces an isomorphism
\[
\lmap{(B_{\ast}\iota_{K})^{\ast}}{H^{\bullet}(B_{\ast}G;\R)}{H^{\bullet}(B_{\ast}K;\R)}.
\]
Together with the isomorphism $\lmap{\Psi_{G}}{H_{c}^{\bullet}(G;\R)}{H^{\bullet}(\X_{u};\R)}$ from \eqref{CompactModel2} in Section\,\ref{SubsecVanEst}, this can be used to intertwine the classifiying map $\map{f_{G_{u}}}{\X_{u}}{B_{\ast}K}$ of the canonical $K$-bundle $G_{u} \to \X_{u}$ with the geometric map
\begin{eqnarray*} \label{MapTransferMap}
	\lmap{T_{G}}{H^{\bullet}(B_{\ast}G;\R)}{H_{c}^{\bullet}(G;\R)}.
\end{eqnarray*}

We would like to compare $T_G$ to the universal map $\sigma_G$.

In odd degrees we have $H^{2k+1}(B_{\ast}K;\R) = \{0\}$ by Lemma \ref{LemmaChernWeil}, whence $T_G = \sigma_G$ for trivial reasons.

In even degree we claim that
\begin{eqnarray} \label{PreciseSigns}
	\sigma_{G} =	(-1)^{k} \cdot T_{G}: H^{2k}(B_{\ast}G;\R) \to H_{c}^{2k}(G;\R).
\end{eqnarray}
This refines the statement of Proposition \ref{MainConvenient}. We see from Corollary \ref{NaturalCharacterized} that in order to prove (\ref{PreciseSigns}) it actually suffices to show the following proposition.

\begin{proposition} \label{PropositionExplicitDescriptionOfNaturalTransformation} The diagram
\[
\begin{xy}\xymatrix{
	H^{2k}(B_*G;\R) \ar[d]_{(B_*\iota_\Gamma)^*} \ar[rr]^{(-1)^k \cdot T_G} && H^{2k}_c(G;\R) \ar[d]^{\iota_\Gamma^*} \\
	H^{2k}(B_*\Gamma;\R) \ar[rr]^{\sigma_\Gamma} && H^{2k}(\Gamma;\R) \\
}\end{xy}
\]
commutes. 
\end{proposition}

The proof of this proposition will occupy the rest of this article.

\subsection{Proof of Proposition \ref{PropositionExplicitDescriptionOfNaturalTransformation}} \label{SubSectionProofOfProposition}
	As before, we denote by $M := \Gamma \backslash G/K$ locally symmetric space associated to the fixed auxiliary lattice $\Gamma$ in $G$. We denote by $\map{f_{G_{u}}}{\X_{u}}{B_{\ast}K}$ the classifying map of the $K$-bundle $G_{u} \to \X_{u}$, and by \mbox{$\map{f_{\Ga\backslash G}}{M}{B_{\ast}K}$} the classifying map of the $K$-bundle $\Ga\backslash G \to M$. For every characteristic class $c \in H^{2k}(B_{\ast}K;\R)$ we obtain characteristic classes
\[
c(G_{u}) = f^{\ast}_{G_{u}}c \in H^{2k}(\X_{u};\R) \jump \text{and} \jump c(\Ga\backslash G) = f^{\ast}_{\Ga\backslash G} c \in H^{2k}(M;\R)
\]
of the bundles $G_{u} \to \X_{u}$ and $\Ga\backslash G \to M$, respectively. Now Proposition \ref{PropositionDualityOfCharacteristicClasses} implies
\[
f^{\ast}_{\Ga\backslash G} c = c(\Ga\backslash G) = (-1)^{k} \cdot \Phi_{\Ga} \bigl( c(G_{u}) \bigr) = \Phi_{\Ga} \circ \bigl( (-1)^{k} \cdot f^{\ast}_{G_{u}} \bigr) c,
\]
thus the triangle
\begin{eqnarray} \label{DiagramCommutativeTriangle}
\begin{xy}
\xymatrix{
& H^{2k}(B_{\ast}K;\R) \ar[ld]_{(-1)^{k} \cdot f_{G_{u}}^{\ast} \quad} \ar[rd]^{f_{\Ga\backslash G}^{\ast}} & \\
H^{2k}(\X_{u};\R) \ar[rr]^{\Phi_{\Ga}} && H^{2k}(M;\R)
}
\end{xy}
\end{eqnarray}
commutes. Combining this with Corollary \ref{LatticeInclusion} we obtain a commutative diagram
\begin{eqnarray} \label{DiagramBigCommutativeSquare}
\begin{xy}
\xymatrix{
H^{2k}(B_{\ast}K;\R) \ar[rr]^{(-1)^{k} \cdot f_{G_{u}}^{\ast}} \ar[d]_{f_{\Ga\backslash G}^{\ast}} && \ar[r]^{\Psi_G} H^{2k}(\X_{u};\R) \ar[lld]_{\Phi_{\Ga}} & H^{2k}(G;\R) \ar[d]^{\io_{\Ga}^{\ast}} \\
H^{2k}(M;\R) \ar[rrr]^{i_M} &&& H^{2k}(\Ga;\R)
}
\end{xy}
\end{eqnarray}
where $\mapin{\io_{\Ga}}{\Ga}{G}$ denotes the inclusion. (In order to keep the diagram readable we do not include the de Rham isomorphisms.) The last step in our argument requires the following lemma. We denote by $\map{f_{\X}}{M}{B_{\ast}\Ga}$ the classifying map of the $\Ga$-bundle $\X \to M$.
\begin{lemma}
	The diagram
\begin{eqnarray} \label{DiagramCommutativeSquare}
\begin{xy}
\xymatrix{
H^{2k}(B_{\ast}G;\R) \ar[d]_{(B_{\ast}\io_{\Ga})^{\ast}} \ar[rr]^{(B_{\ast}\iota_{K})^{\ast}} && H^{2k}(B_{\ast}K;\R) \ar[d]^{f_{\Ga\backslash G}^{\ast}} \\
H^{2k}(B_{\ast}\Ga;\R) \ar[rr]^{f_{\X}^{\ast}} && H^{2k}(M;\R)
}
\end{xy}
\end{eqnarray}
commutes.
\end{lemma}

\begin{proof}
	To show that diagram (\ref{DiagramCommutativeSquare}) is commutative it suffices to prove that the square
\[
\begin{xy}
\xymatrix{
B_{\ast}K \ar[rr]^{B_{\ast}\iota_{K}} && B_{\ast}G \\
M \ar[u]^{f_{\Ga\backslash G}} \ar[rr]^{f_{\X}} && B_{\ast}\Ga \ar[u]_{B_{\ast}\io_{\Ga}}
}
\end{xy}
\]
is commutative up to homotopy. So we have to show that there is a homotopy
\[
B_{\ast}\io_{K} \circ f_{\Ga\backslash G} \,\simeq\, B_{\ast}\io_{\Ga} \circ f_{\X}
\]
between maps from $M$ to $B_{\ast}G$. Since homotopy classes of maps from $M$ to $B_{\ast}G$ are in one-to-one correspondence with isomorphism classes of $G$-bundles over $M$ this is in turn equivalent to the existence of an isomorphism of $G$-bundles
\begin{eqnarray} \label{MapPullBackBundlesIsomorphic}
	(B_{\ast}\io_{K} \circ f_{\Ga\backslash G})^{\ast}EG \,\cong\, (B_{\ast}\io_{\Ga} \circ f_{\X})^{\ast}EG
\end{eqnarray}
over $M$. Now we see from the pullback diagrams
\[
\begin{xy}
\xymatrix{
	(\Ga\backslash G) \xop_{K} G \ar[r] \ar[d] & E_{\ast}K \xop_{K} G \ar[r]\ar[d] & E_{\ast}G \ar[d] \\
	M \ar[r]^{f_{\Ga\backslash G}} & B_{\ast}K \ar[r]^{B_{\ast}\io_{K}} & B_{\ast}G
}
\end{xy}
\]
and
\[
\begin{xy}
\xymatrix{
	\Ga \backslash (\X \times G) \ar[r] \ar[d] & E_{\ast}\Ga \xop_{\Ga} G \ar[r] \ar[d] & E_{\ast}G \ar[d] \\
	M \ar[r]^{f_{\X}} & B_{\ast}\Ga \ar[r]^{B_{\ast}\io_{\Ga}} & B_{\ast}G
}
\end{xy}
\]
that 
\[
(B_{\ast}\io_{K} \circ f_{\Ga\backslash G})^{\ast}EG \,\cong\, (\Ga\backslash G) \xop_{K} G \jump \text{and} \jump (B_{\ast}\io_{\Ga} \circ f_{\X})^{\ast}EG \,\cong\, \Ga \backslash (\X \xop G)
\]
as $G$-bundles over $M$. Here, in the second diagram the quotient $\Ga \backslash (\X \times G)$ is taken with respect to the diagonal action induced by the standard left action of $\Ga$ on $\X=G/K$ and $G$. Thus (\ref{MapPullBackBundlesIsomorphic}) is a consequence of the following fact:

\begin{claim*}
The bundles $(\Ga \backslash G) \xop_{K} G$ and $\Ga \backslash (\X \times G)$ are isomorphic as $G$-bundles over $M$.
\end{claim*}

To prove the claim we write down the isomorphism explicitly. Let us use the notation
\[
\lmap{\pi_{1}}{P_{1} \deq (\Ga \backslash G) \xop_{K} G}{M}, \jump [\Ga g_{1},g_{2}] \mapsto \Ga g_{1}K
\]
and
\[
\lmap{\pi_{2}}{P_{2} \deq \Ga \backslash ((G/K) \times G)}{M}, \jump [g_{1}K,g_{2}] \mapsto \Ga g_{1}K
\]
for the two bundles. Then the map
\[
\lmap{\vphi}{G \xop G}{G \xop G}, \jump (g_{1},g_{2}) \mapsto (g_{1}, g_{1}g_{2}).
\]
descends to a map $\map{\overline{\vphi}}{P_{1}}{P_{2}}$. Similarly, the inverse
\[
\lmap{\vphi^{-1}}{G \xop G}{G \xop G}, \jump (g_{1},g_{2}) \mapsto (g_{1},g_{1}^{-1}g_{2})
\]
descends to a map $P_{2} \to P_{1}$. This proves the claim and finishes the proof of the lemma.
\end{proof}
Attaching diagram (\ref{DiagramCommutativeSquare}) to diagram (\ref{DiagramBigCommutativeSquare}) from the left, we obtain a commutative diagram
\[
\begin{xy}
\xymatrix{
H^{2k}(B_{\ast}G;\R) \ar[r]^{(B_{\ast}\io_{K})^{\ast}} \ar[d]_{(B_{\ast}\io_{\Ga})^{\ast}} & H^{2k}(B_{\ast}K;\R) \ar[rr]^{(-1)^{k} \cdot f_{G_{u}}^{\ast}} && \ar[r]^{\Psi_G} H^{2k}(\X_{u};\R) & H^{2k}(G;\R) \ar[d]^{\io_{\Ga}^{\ast}} \\
H^{2k}(B_{\ast}\Ga;\R) \ar[rrrr]_{\sigma_\Ga} &&&& H^{2k}(\Ga;\R)
}
\end{xy}
\]
By definition of the geometric map $T_{G}$, the upper row coincides with $(-1)^k \cdot T_G$. This finishes the proof of Proposition \ref{PropositionExplicitDescriptionOfNaturalTransformation}.

\bibliographystyle{amsplain}

\begin{thebibliography}{10}

\bibitem{Guido}
Guido's book of conjectures.
\newblock {\em Enseign. Math. (2)}, 54(1-2):3--189, 2008.
\newblock A gift to Guido Mislin on the occasion of his retirement from ETHZ,
  June 2006, Collected by Indira Chatterji.

\bibitem{Borel}
A.~Borel.
\newblock Sur la cohomologie des espaces fibr\'es principaux et des espaces
  homog\`enes de groupes de {L}ie compacts.
\newblock {\em Ann. of Math. (2)}, 57:115--207, 1953.

\bibitem{BorelLNM}
A.~Borel.
\newblock {\em Topics in the homology theory of fibre bundles}, volume 1954 of
  {\em Lectures given at the University of Chicago}.
\newblock Springer-Verlag, Berlin, 1967.

\bibitem{BoWa}
A.~Borel and N.~Wallach.
\newblock {\em Continuous cohomology, discrete subgroups, and representations
  of reductive groups}, volume~67 of {\em Mathematical Surveys and Monographs}.
\newblock American Mathematical Society, Providence, RI, second edition, 2000.

\bibitem{Bott}
R.~Bott.
\newblock Some remarks on continuous cohomology.
\newblock In {\em Manifolds---Tokyo 1973 (Proc. Internat. Conf., Tokyo, 1973)},
  pages 161--170. Univ. Tokyo Press, Tokyo, 1975.

\bibitem{Michelle}
M.~Bucher-Karlsson.
\newblock Finiteness properties of characteristic classes of flat bundles.
\newblock {\em Enseign. Math. (2)}, 53(1-2):33--66, 2007.

\bibitem{Surface}
M.~Burger, A.~Iozzi, and A.~Wienhard.
\newblock Surface group representations with maximal {T}oledo invariant.
\newblock {\em Ann. of Math.}, to appear.

\bibitem{BuMo}
M.~Burger and N.~Monod.
\newblock Bounded cohomology of lattices in higher rank {L}ie groups.
\newblock {\em J. Eur. Math. Soc. (JEMS)}, 1(2):199--235, 1999.

\bibitem{Cartan2}
H.~Cartan.
\newblock La transgression dans un groupe de {L}ie et dans un espace fibr\'e
  principal.
\newblock In {\em Colloque de topologie (espaces fibr\'es), {B}ruxelles, 1950},
  pages 57--71. Georges Thone, Li\`ege, 1951.

\bibitem{DupontLNM}
J.~L. Dupont.
\newblock {\em Curvature and characteristic classes}.
\newblock Lecture Notes in Mathematics, Vol. 640. Springer-Verlag, Berlin,
  1978.

\bibitem{Dupont}
J.~L. Dupont.
\newblock Bounds for characteristic numbers of flat bundles.
\newblock In {\em Algebraic topology, Aarhus 1978 (Proc. Sympos., Univ. Aarhus,
  Aarhus, 1978)}, volume 763 of {\em Lecture Notes in Math.}, pages 109--119.
  Springer, Berlin, 1979.

\bibitem{DupontTopology}
J.~L. Dupont.
\newblock Simplicial de {R}ham cohomology and characteristic classes of flat
  bundles.
\newblock {\em Topology}, 15(3):233--245, 1976.

\bibitem{Gromov}
M.~Gromov.
\newblock Volume and bounded cohomology.
\newblock {\em Inst. Hautes \'Etudes Sci. Publ. Math.}, (56), 1982.

\bibitem{HirzAutom}
F.~Hirzebruch.
\newblock Automorphe {F}ormen und der {S}atz von {R}iemann-{R}och.
\newblock In {\em Symposium internacional de topolog\'\i a algebraica
  International symposium on algebraic topology}, pages 129--144. Universidad
  Nacional Aut\'onoma de M\'exico and UNESCO, Mexico City, 1958.

\bibitem{Hopf41}
H.~Hopf.
\newblock \"{U}ber die {T}opologie der {G}ruppen-{M}annigfaltigkeiten und ihre
  {V}erallgemeinerungen.
\newblock {\em Ann. of Math. (2)}, 42:22--52, 1941.

\bibitem{Helgason}
S.~Helgason.
\newblock Differential geometry, {L}ie groups, and symmetric spaces.
\newblock {\em{Graduate Studies in Mathematics}}, 34, Corrected reprint of the 1978 original, American Mathematical Society, Providence, RI, 2001.

\bibitem{KamberTondeur}
F.~Kamber and P.~Tondeur.
\newblock {\em {Flat manifolds.}}
\newblock {Lecture Notes in Mathematics, No. 67, Springer-Verlag, Berlin}, 1968.

\bibitem{Knapp}
A.~W. Knapp.
\newblock {\em {Lie groups beyond an introduction. 2nd ed.}}
\newblock {Progress in Mathematics 140, Birkh\"auser, Boston}, 2002.

\bibitem{Milnor}
J.~Milnor.
\newblock Construction of universal bundles. {II}.
\newblock {\em Ann. of Math. (2)}, 63, 1956.

\bibitem{BC}
N.~Monod.
\newblock {\em Continuous bounded cohomology of locally compact groups}, volume
  1758 of {\em Lecture Notes in Mathematics}.
\newblock Springer-Verlag, Berlin, 2001.

\bibitem{MonodSurvey}
N.~Monod.
\newblock An invitation to bounded cohomology.
\newblock In {\em International {C}ongress of {M}athematicians. {V}ol. {II}},
  pages 1183--1211. Eur. Math. Soc., Z\"urich, 2006.

\bibitem{Mostow}
M.~A. Mostow.
\newblock Continuous cohomology of spaces with two topologies.
\newblock {\em Mem. Amer. Math. Soc.}, 7(175), 1976.

\bibitem{Tu}
J.-L. Tu.
\newblock Groupoid cohomology and extensions.
\newblock {\em Trans. Amer. Math. Soc.}, 358(11), 2006.

\end{thebibliography}

\end{document}